\documentclass[11pt]{amsart}

\usepackage{graphicx, enumerate, url}
\usepackage{amssymb,mathtools}
\usepackage{algorithm}
\usepackage{algpseudocode}

\numberwithin{equation}{section}
\numberwithin{algorithm}{section}

\theoremstyle{plain}
\newtheorem{theorem}{Theorem}[section]

\newtheorem{lemma}[theorem]{Lemma}
\newtheorem{corollary}[theorem]{Corollary}

\theoremstyle{definition}

\theoremstyle{remark}

\DeclareMathOperator{\gl}{GL}

\newcommand{\N}{\mathbb{N}}
\newcommand{\C}{\mathbb{C}}
\newcommand{\F}{\mathbb{F}}

\newcommand{\R}{\mathbb{R}}

\newcommand{\x}{\mathbf{x}}

\newcommand{\y}{\mathbf{y}}

\newcommand{\cD}{\mathcal{D}}

\newcommand{\cN}{\mathcal{N}}

\newcommand{\bi}{\mathbf{i}}

\newcommand{\bu}{\mathbf{u}}
\newcommand{\bv}{\mathbf{v}}
\newcommand{\bw}{\mathbf{w}}

\newcommand{\by}{\mathbf{y}}

\newcommand{\0}{\mathbf{0}}
\newcommand{\1}{\mathbf{1}}

\begin{document}
\title[Infinimum of a matrix norm of $A$]
{Infimum of a matrix norm of $A$ induced \\by an absolute  vector norm}
\author{Shmuel Friedland}
\date{May 3, 2021}
\begin{abstract}  
We characterize the infimum of a matrix norm of a square matrix $A$ induced by an absolute norm, over the fields of real and complex numbers.  Usually this infimum is greater than the spectral radius of $A$. If $A$ is sign equivalent to a nonnegative matrix $B$ then this infimum is the spectral radius of $B$.
\end{abstract}
\maketitle
 \noindent {\bf 2020 Mathematics Subject Classification.} 15A42, 15A60, 15B48.

\noindent \emph{Keywords}:  Absolute norm, matrix norm, nonnegative matrices, spectral radius.
\maketitle

\section{Introduction} \label{sec:intro}
Let $\F$ be the field of real or complex numbers $\R$ and $\C$ respectively.  Denote by $\F^n$ and $\F^{n\times n}$ the vectors spaces column vectors and the matrices respectively.  Assume that $\|\cdot\|_{\F}:\F^n\to [0,\infty)$ is a norm on $\F^n$. 
Then $\|\x\|_\F$ induces an operator norm on $\F^{n\times n}$, namely 
\begin{equation*}
\|A\|_\F=\max\{\|A\x\|_\F, \x\in\F^n,\|\x\|_\F\le 1\}. 
\end{equation*}

For $A=[a_{ij}]\in\F^{n\times n}$ let $\rho(A)$ be the spectral radius of $A$, which is  maximum modulus of all real and complex eigenvalues of $A$.
The following inequality is well known over the complex numbers:
\begin{equation}\label{specupbd}
\rho(A)\le \|A\|_\F \textrm{ for }A\in\F^{n\times n}.
\end{equation}
We will show that this inequality holds also over real numbers.

Denote by $\gl(n,\F)\subset \F^{n\times n}$ the group of invertible matrices.  Hence  the above inequality yields:
\begin{equation}\label{imspecupbd}
\rho(A)\le \inf\{\|PAP^{-1}\|_\F, \,P\in\gl(n,\F)\}.
\end{equation}
(For $\F=\C$ this result is well known.)

The following result of the author characterizes the norms over $\C^n$ for which equality in \eqref{imspecupbd}  holds for all $A\in\C^{n\times n}$ \cite{Fri79}.
A norm $\|\x\|_\F$ is called an absolute norm if $\|\x\|_\F=\||\x|\|_\F$, where $|(x_1,\ldots,x_n)^\top|:=(|x_1|,\ldots,|x_n|)^\top$ and $|A|=[|a_{ij}|]$.   Recall that an absolute norm is monotone: $\|\x\|_\F\le \|\y\|_\F$ if $|\x|\le |\y|$ \cite[Theorem 7.1.3]{Frib}.

We call $\|A\|_\F$ a matrix absolute norm if $\|\x\|_\F$ is an absolute  norm.  
A norm $\|\x\|_\F$ is called a transform absolute norm if there exists  $P\in\gl(n,\F)$ and an absolute norm $\nu(\x)$ such that $\|\x\|=\nu(P\x)$.
Theorem 3 in \cite{Fri79} shows equality in \eqref{imspecupbd} holds for all $A\in\C^{n\times n}$ if and only if the norm $\|\x\|_\C$ is an absolute transform norm.  In particular we deduce the 
well known approximation result due to Householder \cite[Theorem 4.4]{Hou58}: For each matrix $A\in \C^{n\times n}$ and $\varepsilon>0$ there exists a norm on $\C^n$ such that
\begin{equation}\label{epsupbd}
\|A\|_\C\le \rho(A)+\varepsilon.
\end{equation}
From the proof of this result in \cite{Hou58} or \cite[(2), Section 2.3]{Hou64} it follows that $\|\cdot\|_\C$ can be chosen a transform absolute norm.  (In \cite{Fri79} it was claimed erroneously that the proof of \eqref{epsupbd} yields that  $\|\cdot\|_\C$ can be chosen an absolute norm.)  

Jamal Najim asked if one can choose an absolute norm $\|\cdot\|_\C$ such that \eqref{epsupbd} holds \footnote{Personal communication.}.   We show that this is false even for $2\times 2$ real matrices.

The aim of this note to give a necessary and sufficient conditions when there exists an absolute norm on $\F^n$ such that for a given $A\in\F^{n\times n}$ and $\varepsilon>0$ one has the inequality
\begin{equation}\label{epsupbdF}
\|A\|_\F\le \rho(A)+\varepsilon.
\end{equation}

This condition can be stated as follows.  Let $\|\x\|_2$ be the Euclidean norm on $\F^n$.  Denote by $\cD_n(\F)\subset \gl(n,\F)$ the subgroup of diagonal matrices 
whose absolute value of diagonal elements is $1$.  For a positive integer $m$ let $[m]=\{1,\ldots,m\}$.
Then \eqref{epsupbdF} holds if and only if
\begin{equation}\label{necsufcon}
\begin{aligned}
\lim_{m\to\infty}\max\{
(\rho(A)+\varepsilon)^{-k}\|AD_1A\cdots D_{k-1} A D_k\x\|_2, \\D_1,\dots,D_k\in\cD_n(\F),\x\in\F^n,\|\x\|_2=1, k\in[m]\}<\infty.
\end{aligned}
\end{equation}

Equivalently, we characterize the following infimum: Let $\cN(n,\F)$ be a set of absolute norms on $\F^n$.  Denote
\begin{equation}\label{defmuA}
\begin{aligned}
\mu(A)=\limsup_{k\to\infty}\big(\max\{  \|AD_1A\cdots D_{k-1} A D_k\x\|_2 \\D_1,\dots,D_k\in\cD_n(\F),\x\in\F^n,\|\x\|_2=1\}\big)^{1/k}.
\end{aligned}
\end{equation}
Then 
\begin{equation}\label{mincharmuA}
\mu(A)=\inf\{\|A\|_\F, \, \|\cdot \|_\F\in \cN(n,\F)\}.
\end{equation}

We now survey briefly the results of this paper.  In Section \ref{sec:specupbdl} we prove the inequality \eqref{specupbd}.  In Section \ref{sec:main} we prove the inequality \eqref{mincharmuA}. In Section \ref{sec:addres} we prove the inequality
$\mu(A)\le \rho(|A|)$.  Equality holds if $A$ is sign equivalent to $|A|$.
\section{Proof of inequality \eqref{specupbd}}\label{sec:specupbdl}
\begin{lemma}\label{specupbdlem}
Let $\|\x\|_\F$ be a norm on $\F^n$.  Then the inequality \eqref{specupbd}
holds.
\end{lemma}
\begin{proof} Suppose first that $\F=\C$.  Then there exists an eigenvalue of $\lambda\in\C$ of $A$ such that $|\lambda|=\rho(A)$,  Let $\x\in\C^n$ be the corresponding eigenvector satisfying $\|\x\|_\C=1$.  Then
\begin{equation*}
\|A\|_\C\ge \|A\x\|_\C=\|\lambda\x\|_\C=\rho(A).
\end{equation*}
This shows \eqref{specupbd}.

Assume now that $A\in \R^{n\times n}$.  Suppose first that $A$ has a real eigenvalue $\lambda$ such that $|\lambda|=\rho(A)$.  Then the above arguments yield that $\|A\|_\R\ge \rho(A)$.  Suppose that $\rho(A)>0$ and $\lambda=\rho(A)\zeta$ where $\zeta=e^{2\pi\bi\theta}$ where $\theta\in [0,1)$. Suppose first that $\theta$ is a rational number.  Then $\zeta^k=1$ for some positive integer $k$.  That is $\lambda^k$ is a positive eigenvalue of $A^k$.
Recall that $\rho(A^k)=\rho(A)^k$.  Using the previous result we obtain
\begin{equation*}
\|A\|_\R^k\ge \|A^k\|_\R\ge \rho(A^k)=\rho(A)^k.
\end{equation*}
Hence \eqref{specupbd} holds.  

Suppose second that $\theta$ is irrational.  Then the complex eigenvector of $A$ corresponding to $\lambda$ is $\bu+\bi\bv$ can be chosen to satisfy
\begin{equation*}
A^k(\bu+\bi\bv)=\rho(A)^ke^{2\pi\bi k\theta}(\bu+\bi\bv), \quad \|\bu\|=1,\bv\in\R^n\setminus\{\0\}, k\in\N.
\end{equation*}
Hence
\begin{eqnarray*}
\|A^k\bu\|_\R=\rho(A)^k\|(\cos 2\pi k\theta)\bu -(\sin 2\pi k\theta)\bv\|_\R\ge\\ 
\rho(A)^k\big(|\cos 2\pi k\theta|\|\bu\|_\R- |\sin 2\pi k\theta|\|\bv\|_\R\big).
\end{eqnarray*}
Recall that for a given $\delta\in (0,1))$ there exists an infinite sequence  of positive integers $k_l,m_l$, for $l\in\N$, such that $|k_l\theta -m_l|\le \delta/(2\pi)$.  Hence there exists an infinite sequence of positive integers $k_l$ such that 
\begin{eqnarray*}
\|A^{k_l}\bu\|_\R\ge\frac{\rho(A)^{k_l}}{2}\|\bu\|_\R=\frac{\rho(A)^{k_l}}{2}, \quad l\in\N.
\end{eqnarray*}
Therefore
\begin{equation*}
\|A\|_\R \ge (\|A^{k_l}\|_\R)^{1/k_l}\ge (\|A^{k_l}\bu\|_\R)^{1/k_l}\ge \rho(A)2^{-1/k_l}.
\end{equation*}
Letting $l\to\infty$ we deduce that $\|A\|_\R\ge \rho(A)$.

\end{proof}
\section{Proof of the main result}\label{sec:main}
\begin{lemma}\label{upbdprodlem}
Let $\|\x\|_\F$ be an absolute norm on $\F^n$.  Then for any $A\in\F^{n\times n}$
the following relations holds:
\begin{equation}\label{bdsabnorm}
\begin{aligned}
\|D_1 A D_2\|_\F=\|A\|_\F \textrm{ for all } D_1,D_2\in \cD_n(\F),\\
\|AD_1AD_2\cdots AD_{k}\|_\F\le \|A\|_\F^k, \quad k\in\N,\\
\|A\|_\F\ge \max\{\rho(DA), D\in\cD_n(\F)\}.
\end{aligned}
\end{equation}
\end{lemma}
\begin{proof} Assume that $D_1,D_2\in \cD_n(\F)$.
Since $\|\x\|_\F$ is an absolute norm we obtain $\|D_2\x\|_\F=\|\x\|_\F$, and $\|D_2\|_\F=1$.
Hence 
\begin{equation*}
\|D_1AD_2\x\|_\F/\|\x\|_\F=\|AD_2\x\|_\F/\|D_2\x\|_\F \textrm{ for } \x\ne \0.
\end{equation*}
This proves the first equality of \eqref{bdsabnorm}.  Use the submultiplicativity of the norm $\|B\|_\F$ and the first equality \eqref{bdsabnorm} to deduce the second equality of \eqref{bdsabnorm}.   

The inequality \eqref{specupbd} yields $\|A\|_\F=\|D_1A\|_\F\ge \rho(D_1A)$.  This proves the third inequality of \eqref{bdsabnorm}.
\end{proof}
\begin{corollary}\label{countex}  Let 
\begin{equation*}\label{countex1}
A=\left[\begin{array}{rr}1&1\\-1&-1\end{array}\right], D=\left[\begin{array}{rr}1&0\\0&-1\end{array}\right],
B= DA=\begin{bmatrix}1&1\\1&1\end{bmatrix}.
\end{equation*}
Then $\rho(A)=0, \rho(B)=\|A\|_2=2$.  Hence for any absolute norm on $\R^2$ we have the sharp inequality $\|A\|_\R\ge 2$.
\end{corollary}
\begin{theorem}\label{mainthm}
Assume that $A\in \F^{n\times n}$.  Let $\mu(A)$ be defined by \eqref{defmuA}.  Then equality \eqref{mincharmuA} holds.
\end{theorem}
\begin{proof}  Assume that $\|\x\|_\F$ is an absolute norm.
Recall that all norms on $\F^{n\times n}$ are equivalent.  Hence for a given absolute norm $\|\x\|_\F$ one has an inequality
\begin{equation*}
\frac{1}{K(\|\cdot\|_\F)}\|\x\|_\F\le \|\x\|_2\le K(\|\cdot\|_\F)\|\x\|_\F, \x\in\F^n,
\end{equation*}
for some $K(\|\cdot\|_\F)\ge 1$.  Therefore in the definition of $\mu(A)$ given by \eqref{defmuA} we can replace the norm $\|\x\|_2$ by an absolute norm $\|\x\|$,  Use the second inequality of \eqref{bdsabnorm} to obtain
\begin{equation}\label{upbdmuA}
\mu(A)\le \|A\|_\F.
\end{equation}
Let 
\begin{equation*}
\tilde\mu(A)=\inf\{\|A\|_\F, \|\cdot\|_\F\in\cN(n,\F)\}.
\end{equation*}
The inequality \eqref{upbdmuA} yields that $\mu(A)\le \tilde\mu(A)$.

We now show that for each $\varepsilon>0$ there exists an absolute  norm $\|\x\|_\F$ on $\F^n$ such that $\|A\|_\F\le \mu(A)+\varepsilon$.  From the definition of $\mu(A)$ it is straightforward to show that 
\begin{equation}\label{necsufcon1}
\begin{aligned}
\lim_{m\to\infty}\max\{
(\mu(A)+\varepsilon)^{-m}\|AD_1A\cdots D_{m-1} A D_m\x\|_2, \\D_1,\dots,D_m\in\cD_n(\F),\x\in\F^n,\|\x\|_2=1\}=0.
\end{aligned}
\end{equation}
Hence
\begin{equation}\label{necsufcon2}
\begin{aligned}
\lim_{m\to\infty}\max\{
(\mu(A)+\varepsilon)^{-k}\|AD_1A\cdots D_{k-1} A D_k\x\|_2, \\D_1,\dots,D_k\in\cD_n(\F),\x\in\F^n,\|\x\|_2=1, k+1\in[m]\}<\infty.
\end{aligned}
\end{equation}

Define
\begin{equation}\label{necsufcon3}
\begin{aligned}
\|\x\|_\F:=
\lim_{m\to\infty}\max\{
(\mu(A)+\varepsilon)^{-k}\|AD_1A\cdots D_{k-1} A D_k\x\|_2, \\D_1,\dots,D_k\in\cD_n(\F),\x\in\F^n, k+1\in[m]\}<\infty.
\end{aligned}
\end{equation}
(We let the values of $\mu(A)+\varepsilon)^{-k}\|AD_1A\cdots D_{k-1} A D_k\x\|_2$ to be $\|\x\|_2$ and $\|AD_1\x\|_2$ for $k=0$ and $k=1$ respectively.)
Clearly $\|\0\|_\F=0$.
The inequality \eqref{necsufcon2} yields that $\|\x\|_\F<\infty$ for $\x\ne 0$.  Clearly,
$\|t\x\|_\F=|t|\|\x\|_\F$ for $t\in\F$.  
The inequality \eqref{necsufcon1} yields that for each $\x\in\F$ there exists a nonnegative integer $k=k(\x)$ and $D_1,\ldots, D_{k(\x)}\in\cD_n(\F)$ such that  
\begin{equation*}
\|\x\|_\F=\begin{cases}
\|\x\|_2 \textrm{ if } k(\x)=0,\\
(\mu(A)+\varepsilon)^{-1} \|AD_1\x\|_2 \textrm{ if } k(\x)=1,\\
(\mu(A)+\varepsilon)^{-k(\x)}\|AD_1A\cdots D_{k(\x)-1} A D_{k(\x)}\x\|_2  \textrm{ if } k(\x)>1.
\end{cases}
\end{equation*}
The maximum definition of $\|\x\|_\F$ yileds that $\|\x\|_\F$ satisfies the triangle inequality and the equality $\|D\x\|_\F=\|\x\|_\F$ for all $D\in\cD_n(\F)$.  Hence $\|\x\|_\F$ is an absolute norm on $\F^n$.  It is left to show that 
\begin{equation}\label{fundup}
\|A\x\|_\F\le (\mu(A)+\varepsilon)\|\x\|_\F, \quad \x\in\F^n.
\end{equation}
Observe that 
\begin{eqnarray*}
\|A\x\|_\F:=(\mu(A)+\varepsilon)
\lim_{m\to\infty}\max\{
(\mu(A)+\varepsilon)^{-(k+1)}\|AD_1A\cdots D_{k-1} A D_k A\x\|_2, \\D_1,\dots,D_k\in\cD_n(\F),\x\in\F^n, k+1\in[m]\}<\infty.
\end{eqnarray*}
Clearly 
\begin{eqnarray*}
AD_1A\cdots D_{k-1} A D_k A\x=AD_1A\cdots D_{k-1} A D_k A D_{k+1}\x, \quad D_{k+1}=I_n.
\end{eqnarray*}
Hence 
\begin{eqnarray*}
\max\{
(\mu(A)+\varepsilon)^{-(k+1)}\|AD_1A\cdots D_{k-1} A D_k A\x\|_2, \\D_1,\dots,D_k\in\cD_n(\F),\x\in\F^n, k+1\in[m]\}\le\\
\max\{
(\mu(A)+\varepsilon)^{-q}\|AD_1A\cdots D_{q-1} A D_q A\x\|_2, \\D_1,\dots,D_q\in\cD_n(\F),\x\in\F^n, q+1\in[m+1]\}
\end{eqnarray*}
This establishes \eqref{fundup}.  Hence $\|A\|_\F\le \mu(A)+\varepsilon$.
\end{proof}
\begin{theorem}\label{thm2}  Let $A\in\F^{n\times n}$ then there exists an absolute norm $\|\x\|_\F$ and $\varepsilon>0$ such that the inequality \eqref{epsupbdF} holds if and only if the condition \eqref{necsufcon} holds.
\end{theorem}
\begin{proof}  Suppose first that there exists an absolute norm on $\F^n$ such that
 \eqref{epsupbdF} holds.  As in the proof of Theorem \ref{mainthm} we deduce that 
\begin{equation*}
(\rho(A)+\varepsilon)^{-k}\|AD_1\cdots D_{k-1}AD_k\x\|_\F\le(\rho(A)+\varepsilon)^{-k}\|A\|^k_\F\|\x\|_\F\le \|\x\|_\F. 
\end{equation*}
In view of equivalence of norms $\|\x\|_\F$ and $\|\x\|_2$ we deduce
\begin{eqnarray*}
(\rho(A)+\varepsilon)^{-k}\|AD_1\cdots D_{k-1}AD_k\x\|_2\le\\
(\rho(A)+\varepsilon)^{-k}K(\|\cdot\|_\F) \|AD_1\cdots D_{k-1}AD_k\x\|_\F\le\\
K(\|\cdot\|_\F)\|\x\|_\F\le K(\|\cdot\|_\F)^2\|\x\|_2.
\end{eqnarray*}
 Hence the condition \eqref{necsufcon} holds.
 
 Assume now that the condition \eqref{necsufcon} holds.  Define $\|\x\|_\F$
 as in \eqref{necsufcon3} by replacing $(\mu(A)+\varepsilon)$ with $(\rho(A)+\varepsilon)$.  Then the arguments of the proof of Theorem \ref{mainthm}
yield that $\|\x\|_\F$ is an absolute norm for which the inequality \eqref{epsupbdF} holds.
\end{proof}
\section{Additional results and remarks}\label{sec:addres}
A matrix $A\in\F^{n\times n}$ is said to be sign equivalent to $B\in\F^{n\times n}$ if $A=D_1 B D_2$ for some $D_1,D_2\in \cD_n(\F)$.  A matrix $B\in\R^{n\times n}$ is called nonnegative if $B=|B|$.  The following lemma generalizes the example in Corollary \ref{countex}.
\begin{lemma}\label{eqnnmat}  Let $A\in\F^{n\times n}$.  Then
\begin{equation}\label{upbndabsA}
\mu(A)\le \rho(|A|).
\end{equation}
Equality holds if $A$ is sign equivalent to $|A|$.
\end{lemma}
\begin{proof}  Let $B=|A|$ and assume that $A$ is sign equivalent to $B$.
Suppose first that $B$ is an irreducible matrix.  That is $(I_n+B)^{n-1}$ is a positive matrix.  Then Perron-Frobenius theorem \cite{Frib} yields that there exist positive eigenvector $\bu$ and  $\bv$ of $B$ and  $B^\top$ respectively such that
\begin{equation*}
B\bv=\rho(B)\bu, \quad  B^\top\bu=\rho(B)\bu, \quad \|\bv\|_2=1,\bv^\top\bu=1, \rho(A)>0.  
\end{equation*}
For a positive vector $\bw\in\R^n$ 
define an absolute norm on $\F^n$:
\begin{equation}\label{defnunorm}
\nu(\x)=\bw^\top |\x|.
\end{equation}
Assume that $\bw=\bv$.
Then 
\begin{eqnarray*}
\nu(\bu)=\bv^\top \bu=1, \quad \nu(B\bu)=\rho(B)\nu(\bu)=\rho(B),\\
\nu(B\x)=\bv^\top |B\x|\le \bv^\top B|\x|=\rho(B)\bv^\top|\x|=\rho(B)\nu(\x).
\end{eqnarray*}
Hence $\nu(B)=\rho(B)$.  As $\nu$ is absolute we deduce from \eqref{bdsabnorm} that $\nu(A)=\nu(B)=\rho(B)$.
Next observe that  
$B=D_1^{-1}A D_2^{-1}$ is similar to $D_2^{-1}D_1^{-1}A$.  Hence $\rho(D_2^{-1}D_1^{-1}A)=\rho(B)$.  Inequalities \eqref{bdsabnorm} yield that $\|A\|_\F\ge \rho(B)$ for any absolute norm. Hence $\mu(A)=\rho(|A|)$.

Assume that $B$ is a reducible matrix and $\bw>\0$. Let $\nu$ be defined by  \eqref{defnunorm}.  Note that  $\nu$ is a weighted $\ell_1$ norm.  It is straightforward to show that
\begin{equation*}
\nu(B)=\max\{\frac{(B^\top \bw)_i}{w_i}, i\in[n]\}.
\end{equation*}
See \cite[Example 5.6.4]{HJ} for the case $\bw=(1,\ldots,1)^\top$.  Recall the generalized Collatz-Wielandt characterization of the spectral radius of $\rho(B)=\rho(B^\top)$ \cite[Part (1), Theorem 3.2]{Fri20}:
\begin{equation*}
\rho(B)=\inf_{\bw>\0} \max\{\frac{(B^\top \bw)_i}{w_i}, i\in[n]\}.
\end{equation*}
Hence for a given $\varepsilon>0$ there exists $\bw>0$ such that 
\begin{equation*}
\rho(B)\le \nu(B)\le \rho(B)+\varepsilon.
\end{equation*}
Therefore $\mu(A)=\rho(|A|)$.

Assume now that $A$ is not sign equivalent to $|A|$.  Suppose that $\|\x\|_\F$ is an absolute norm.  Observe that $|A\x|\le |A||\x|$.   As $\|\x\|_\F$ is a monotone norm it follows that $\|A\x\|_\F\le \||A||\x|\|_\F$, and $\|A\|_\F\le \||A|\|_\F$ .  Hence inequality \eqref{upbndabsA} holds.
\end{proof}

We close with section with a brief discussion of absolute norms on $\R^n$ and $\C^n$.
Clearly if $\|\x\|_\C$ is an absolute norm on $\C^n$ then the restriction of this norm on $\R^n$ gives an absolute norm $\|\x\|_\R$.  It is also well known that given an absolute norm $\|\x\|_\R$ it induces an absolute norm $\|\x\|_\C$ on $\C^n$  by the equality $\|\x\|=\||\x|\|_\R$.  Indeed, for $\x,\y\in\C^n$ we have that 
$|\x+\y|\le |\x|+|\y|$. The monotonicity of $\|\x\|_\R$ yields 
\begin{eqnarray*}
\|\x+\y\|_\C=\||\x+\y|\|_\R\le \||\x|+|\y|\|_\R\le \||\x|\|_\R+\||\y|\|_\R=\|\x\|_\C+\|\y\|_\C.
\end{eqnarray*}

Assume now that $\|\x\|_\R$ is an absolute norm on $\R^n$.  Let $A\in\R^{n\times n}$.  Then $\|A\|_\C$ is the induced norm by $\|\x\|_\C$.  Clearly
\begin{equation}\label{basmaj}
\|A\|_\R\le \|A\|_\C.
\end{equation}
It is not obvious to the author that one has always equality in the above inequality for a general absolute norm on $\R^n$.
$$\\$$
{\em Acknowledgment}---I thank Jamal Najim for posing the problem that inspired this paper. 
The author was partially supported by Simons collaboration grant for mathematicians.

\bibliographystyle{plain}

\end{document}